\documentclass[12pt,a4paper, oneside]{amsart}

\usepackage{amsmath, nicefrac, amsthm, verbatim, amsfonts, mathtools, amssymb, upgreek, xcolor, bbm}
\usepackage{graphics, xspace, enumerate}
\usepackage{stix}
\usepackage{dsfont}
\usepackage[a4paper,margin=2.5cm]{geometry}
\usepackage[bb=boondox]{mathalfa}

\usepackage{graphicx}
\usepackage[colorlinks=true,citecolor=red,urlcolor=blue,linkcolor=red,bookmarksopen=true,unicode=true,pdffitwindow=true]{hyperref}
\usepackage[english]{babel}
\usepackage[languagenames,fixlanguage]{babelbib}
\hypersetup{pdfauthor={}}
\hypersetup{pdftitle={}}

\theoremstyle{plain}
\newtheorem{theorem}{Theorem}[section]
\newtheorem{lemma}[theorem]{Lemma}
\newtheorem{proposition}[theorem]{Proposition}
\newtheorem{remark}[theorem]{Remark}

\title[Skew-product decomposition of Brownian motion on ellipsoid]{Skew-product decomposition of Brownian motion on ellipsoid}

\author[I.\ Valenti\'c]{Ivana\ Valenti\'c}
\address[Ivana\  Valenti\'c]{
	Department of Mathematics\\University of Zagreb\\ Croatia and Department of Statistics\\ University of Warwick\\ UK}
\email{ivana.valentic@math.hr}


\begin{document}
	\allowdisplaybreaks[4]
	
	\begin{abstract}
		In this article we obtain a skew-product decomposition of a Brownian motion on an ellipsoid of dimension $n$ in a Euclidean space of dimension $n+1$. We only consider such ellipsoid whose restriction to first $n$ dimensions is a sphere and its last coordinate depends on a variable parameter. We prove that the projection of this Brownian motion on to the last coordinate is, after a suitable transformation, a Wright-Fisher diffusion process with atypical selection coefficient.  
	\end{abstract}
	
	\maketitle

\section{Introduction}

Brownian motion (BM) is a process of fundamental importance in natural sciences and in applications one often considers BM on curved surfaces and other manifolds (see \cite{Faraudo-2002}). In this paper we are interested in a BM on hyperellipsoid, that is, an ellipsoid of dimension $n\geq 2$ in a Euclidean space of dimension $n+1$, given by equation

\begin{equation*}
	x_1^2+x_2^2+\ldots+x_n^2+\frac{y^2}{c^2}=1,
\end{equation*}
and denoted $\mathbb{E}^{n}(c)$ where $c>0$ is a fixed constant. In \cite{Mijatovic-Mramor-Uribe-2018} authors investigate the process obtained by projecting BM on a sphere to the ball of lower dimensions. They give a complete
characterization of such processes in terms of SDEs they satisfy. We use a similar approach to obtain a skew product decomposition of a BM on $\mathbb{E}^{n}(c)$. In the case of a sphere this was done in \cite{Mijatovic-Mramor-Uribe-2020}. There authors also observe that a linear transformation of the process obtained in such decomposition is a Wright-Fisher diffusion process with fixed mutation coefficients. The aim of this paper is to investigate how changing the geometry will effect the change in the complexity of a Wright-Fisher diffusion obtained by similar decomposition. We have decided to concentrate on changing the geometric properties only on the last coordinate because the Wright-Fisher diffusion arises as a projection onto one coordinate and therefore to study its properties it makes most sense to keep the geometry in the remaining dimensions as simple as possible in order to not complicate the calculations any further.

The first step is a construction of BM on $\mathbb{E}^{n}(c)$. According to Theorem 3.1.4. in \cite{Hsu-2002} BM $(Z_t)_{t\ge 0}$, $Z=(X^1,X^2,\ldots, X^n, Y)^\top$, on ellipsoid $\mathbb{E}^{n}(c)$ is given (in Itô form) by following SDE on $\mathbb{R}^{n+1}$, which possesses the unique strong solution,
\begin{equation}\label{SDE}
	dZ_t=\sigma(Z_t)dB_t+b(Z_t)dt, \hspace{1cm} Z_0 \in \mathbb{E}^{n}(c)
\end{equation}
where $(B_t)_{t\ge 0}$, $B=(B^1,B^2,\ldots,B^{n+1})^\top$, is a BM on $\mathbb{R}^{n+1}$, diffusion coefficient is given by
\begin{equation*}
	\sigma(z)= I-\frac{1}{r^2}\begin{pmatrix}
		c^4xx^\top & c^2yx\\
		c^2yx^\top & y^2\\
	\end{pmatrix}
\end{equation*}
and drift coefficient is given by
\begin{equation*}\label{b}
	b(z)= \begin{pmatrix}
		b(x)\\
		b(y)\\
	\end{pmatrix}= -\frac{1}{2}
	\begin{pmatrix}
		x\frac{c^4((n-1)r^2+c^2)}{r^4}\\
		y\frac{c^2((n-1)r^2+c^2)}{r^4}\\
	\end{pmatrix},
\end{equation*}
where $r:=\sqrt{|x|^2c^4+y^2}=\sqrt{y^2+c^4-y^2c^2}=\sqrt{(c^4-c^2)|x|^2+c^2}$ and $z=(x^\top,y)^\top=(x^1,x^2,\ldots,x^n,y)^\top$.
\begin{remark}
	In a special case when $c=1$, that is when the ellipsoid is a sphere, we have
	\begin{equation*}
		\sigma(z)= I-zz^\top \,\,\, \text{ and } \,\,\, b(z)=-\frac{n}{2}z^\top,
	\end{equation*}
	which is the Stroock representation of spherical BM.
\end{remark}
We introduce the following time change: $S_t:= \int_0^t \frac{c^2}{c^2-Y_s^2}ds$. In Lemma \ref{lema} we prove that it is well defined. Clearly it is continuous, strictly increasing, and $\lim_{t \to \infty} S_t=\infty$. We denote its inverse by $T:\left[0,\infty\right> \to \left[0, \infty\right>$. As the main result of this paper we prove the following
\begin{proposition}[Skew product decomposition of BM on $\mathbb{E}^{n}(c)$]\label{SPD}
	Let $(Z_t)_{t\ge 0}=((X_t^\top,Y_t)^\top)_{t\ge 0}$ be a solution of
	SDE \eqref{SDE} such that $X_0\neq 0$. Then 
	\begin{itemize}
		\item[(i)] $(\hat{V}_t)_{t\ge 0}$ given by $\hat{V}_t=\frac{X_{T_t}}{|X_{T_t}|}$ is a BM on the sphere $\mathbb{S}^{n-1}$ independent of $(Y_t)_{t\ge 0}$,
		\item[(ii)] there exists an invertible function $f:[-c,c]\to[0,1]$ such that process $(F_t)_{t \ge 0}$, given by $F_t=f(Y_t)$, is a  Wright-Fisher diffusion process with atypical selection coefficient, that is it satisfies the following SDE
		\begin{equation*}
			dF_t=h(c)\sqrt{F_t\left(1-F_t) \right)}d\Tilde{B}_t+\gamma(F_t)dt,
		\end{equation*}
		where $(\Tilde{B}_t)_{t\ge 0}$ is a scalar BM independent of $(\hat{V}_t)_{t\ge 0}$, $h:\left<0,\infty\right> \to \left<0,\infty\right>$ is a  deterministic function and drift coefficient $\gamma:[0,1]\to \mathbb{R}$ is given by
		\begin{equation}\label{gama}
			\gamma(\xi)=\left(\frac{h^2(c)}{4}(1-\xi)-\frac{h^2(c)}{4}\xi\right)-\frac{ch(c)(n-1)}{2} r(\xi)\xi(1-\xi),
		\end{equation}
		where $r:\left<0,1\right>\to \mathbb{R}$ is an analytic function.
	\end{itemize}
	Hence we obtain the skew-product decomposition $X_t=\sqrt{1-\frac{|Y_{S_t}|^2}{c^2}}\hat{V}_{S_t}$, where $Y_t=f^{-1}(F_t)$.
\end{proposition}
We proceed as follows. In the second section we determine the SDE for process $(Y_t)_{t \ge 0}$, we show that it visits neither $c$ nor $-c$ and prove the (ii) part of Proposition \ref{SPD}. In the third section we determine the SDE for process $(X_t)_{t \ge 0}$ and use it to prove part (i) of Proposition \ref{SPD}.

\section{SDE for the last coordinate}
From \eqref{SDE}, using the notation $R_t:=\sqrt{|X_t|^2c^4+Y_t^2}=\sqrt{Y_t^2+c^4-Y_t^2c^2}$, we get that $(Y_t)_{t \ge 0}$ satisfies the following
\begin{equation*}
	dY_t=\sqrt{1-\frac{Y_t^2}{R_t^2}} A_t dB_t - \frac{Y_t}{2}\frac{c^2((n-1)R_t^2+c^2)}{R_t^4}dt,
	\hspace{0.2cm}\text{where}\hspace{0.2cm}
	A_t:=\begin{pmatrix}
		\frac{-c^2Y_t}{R_t^2\sqrt{1-\frac{Y_t^2}{R_t^2}}}X_t^\top, &  \sqrt{1-\frac{Y_t^2}{R_t^2}} \\
	\end{pmatrix}.
\end{equation*}
Let $(\Tilde{B}_t)_{t \ge 0}$ be defined by $\Tilde{B}_t=\int_0^t A_s dB_s$. Since $A_t A_t^\top =1$ implies that $\langle \Tilde{B} \rangle_t=\int_0^t A_s A_s^\top ds= t$
by Levy’s characterization we conclude that $(\Tilde{B}_t)_{t \ge 0}$ is a standard scalar BM. We conclude that $(Y_t)_{t\ge 0}$ satisfies the following SDE
\begin{equation}\label{final3SDE}
	dY_t=\sqrt{1-\frac{Y_t^2}{Y_t^2+c^4-Y_t^2c^2}} d\Tilde{B}_t- \frac{Y_t}{2}\frac{c^2((n-1)\left( Y_t^2+c^4-Y_t^2c^2\right)+c^2)}{\left( Y_t^2+c^4-Y_t^2c^2\right)^2}dt.
\end{equation}
\begin{lemma}\label{lema}
	If $Y_0\neq -c,c$ then process $(Y_t)_{t \ge 0}$ never hits $-c$ and $c$.
\end{lemma}
\begin{proof}
	According to Lemma 6.1(ii) in \cite{Karlin-Taylor-1981} to prove that $(Y_t)_{t \ge 0}$ never hits $c$ it is enough to show that $\int_0^cs(\xi)d\xi=\infty$, where
	\begin{equation*}
		-\ln s(\xi)= \int^\xi \frac{-y((n-1)(y^2+c^4-y^2c^2)+c^2)}{(y^2+c^4-y^2c^2)(c^2-y^2)}dy.
	\end{equation*}
	Since $s(\xi)=\frac{1}{c}\sqrt{c^4-c^2\xi^2+\xi^2}(c^2-\xi^2)^{-\frac{n}{2}}$ and $\sqrt{c^4-c^2\xi^2+\xi^2}\geq \min \left\{ c,c^2\right\}$ for $\xi \in [-c,c]$, we have
	$\int_0^cs(\xi)d\xi \geq \frac{1}{c}\min \left\{ c,c^2\right\} \int_0^c(c^2-\xi^2)^{-\frac{n}{2}}d\xi=\infty$
	for $\frac{n}{2}\geq 1$ which completes the proof. 
\end{proof}
\begin{remark}
	In the special case when $c=1$, that is when the ellipsoid is a sphere, we have
	\begin{equation*}
		dY_t=\sqrt{1-Y_t^2} d\Tilde{B}_t- \frac{n}{2}Y_tdt.
	\end{equation*}
	In \cite{Mijatovic-Mramor-Uribe-2020} it was first observed that $(Y_t)_{t \ge 0}$ is a linear transformation of a Wright-Fisher diffusion. In our case this transformation will be a more complicated function. It still make sense to look for such transformation because from \eqref{final3SDE} we see that $(Y_t)_{t \ge 0}$ is, analogously to the Wright-Fisher diffusion, a diffusion process on a compact state space with a positive diffusion coefficient which dies out on the boundary.
\end{remark}
We wish to find a function $f:[-c,c]\to[0,1]$ such that $(f(Y_t))_{t \ge 0}$ is a Wright-Fisher diffusion that is it satisfies the following SDE
\begin{equation*}
	df(Y_t)=h(c)\sqrt{f(Y_t)\left(1-f(Y_t) \right)}d\Tilde{B}_t+\gamma(f(Y_t))dt.
\end{equation*}
By Ito's formula we need function $f$ to satisfy the following nonlinear ODE
\begin{equation}\label{PDE}
	f'(\xi)\sqrt{\frac{c^4-\xi^2c^2}{\xi^2+c^4-\xi^2c^2}}=h(c)\sqrt{f(\xi)\left(1-f(\xi) \right)}.
\end{equation}

To solve \eqref{PDE}, consider the incomplete elliptic integral of the second kind, given by $E(\phi \mid m):= \int_0^{\phi} \sqrt{1-m\sin^2(\xi)} d\xi$. We will denote by $E\left(1-\frac{1}{c^2}\right)=E\left(\frac{\pi}{2} \mid 1-\frac{1}{c^2}\right)$ the complete elliptic integral of the second kind. Its name comes from the fact that for an ellipse with semi-major axis $1$ and semi-minor axis $\frac{1}{c}$ the circumference of the ellipse is $4$ times the complete elliptic integral of the second kind $E(1-\frac{1}{c^2})$. Clearly, $E(\phi \mid m)$ is strictly increasing, infinitely differentiable function of $\phi$ and it follows from \cite{Vel-1969} that it is analytical. For more information on elliptic integrals see Chapter $19$ in \cite{Olver-Lozier-Boisvert-Clark-2010}.

For $h$ such that $h(c):=\frac{\pi}{2cE\left(1-\frac{1}{c^2}\right)}$ the solution $f$ of the equation \eqref{PDE} is the well defined invertible function given by
\begin{equation*}
	f(\xi)= \cos^2(\frac{\pi}{4}-\frac{c}{2}h(c)E\left(\arcsin(\frac{\xi}{c}) \mid 1-\frac{1}{c^2}\right)).
\end{equation*}
Using Ito's formula we conclude that $\gamma$ satisfies \eqref{gama} where 
\begin{equation*}
	r(\xi)=\frac{f^{-1}(\xi)}{ \sqrt{(f^{-1}(\xi)^2+c^4-f^{-1}(\xi)^2c^2)(c^2-f^{-1}(\xi)^2)\xi(1-\xi)}}
\end{equation*} 
\begin{figure}[h!]
	\centering
	\includegraphics[width=0.3\textwidth]{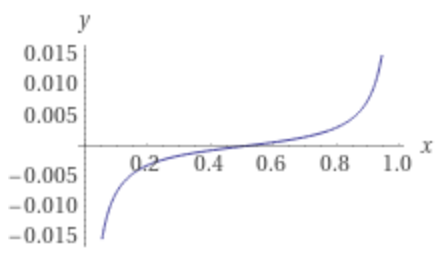} \hspace{2cm}
	\includegraphics[width=0.3\textwidth]{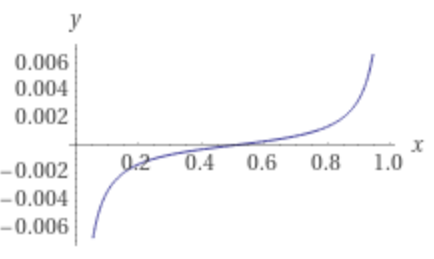}
	\caption{Graph of function approximating $\xi \mapsto r(\xi)$ for $c=20$ and $c=30$}
\end{figure}

Notice that $r$ is a product and composition of functions which are analytic on $\left<0,1\right>$ and therefore $r:\left<0,1\right>\to \mathbb{R}$ is analytic function. For the boundary points we have $\lim \limits_{\xi \to 0} r(\xi) = -\infty$ and $\lim \limits_{\xi \to 1} r(\xi) = \infty$.

Let us specifically determine the behaviour of the process $(F_t)_{t \ge 0}$ at the boundary, that is for $Y_t \in \left\{-c,c \right\}$. The diffusion coefficient vanishes and for the drift we get
\begin{equation*}
	\lim \limits_{\xi \to c} \gamma(f(\xi))=-\frac{nh^2(c)}{4} \hspace{0.5cm}\text{and}\hspace{0.5cm}
	\lim \limits_{\xi \to -c} \gamma(f(\xi))=\frac{nh^2(c)}{4}.
\end{equation*}

\begin{figure}[h!]\label{grafh}
	\centering
	\includegraphics[width=0.3\textwidth]{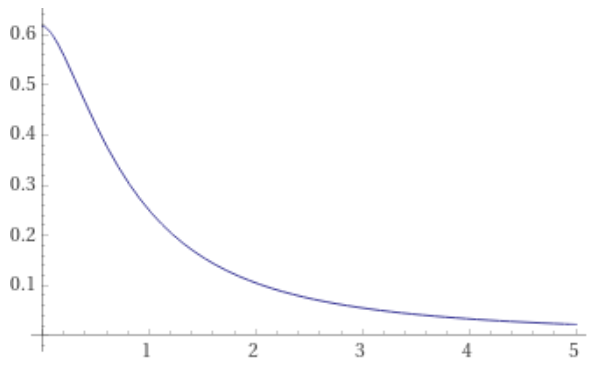}
	\caption{Graph of function $c \mapsto \frac{h^2(c)}{4}$}
\end{figure}
As seen in the Figure \ref{grafh} $\lim \limits_{c \to \infty} \frac{h^2(c)}{4}=0$ and $\lim \limits_{c \to 0} \frac{h^2(c)}{4}=\frac{\pi^2}{16}\approx 0.62$.
\begin{remark}
	In the case of a sphere for equation \eqref{PDE} we get $\sqrt{1-\xi^2}f'(\xi)=\sqrt{f(\xi)\left(1-f(\xi) \right)}$ for which the solution is a linear function $f(\xi)=\frac{\xi\pm 1}{2}$. Therefore $(F_t)_{t \ge 0}$ is a Wright-Fisher diffusion that satisfies the following SDE
	\begin{equation*}
		dF_t=\sqrt{F_t\left(1-F_t \right)} d\Tilde{B}_t+n \frac{1-2F_t}{4}dt.
	\end{equation*}
\end{remark}

We see that a small change in the geometry of the state space of BM results in a rather large change in the level of complexity of the projection. The transformation needed to get the Wright-Fisher diffusion is in this general case more complicated than a linear function, but the Wright-Fisher diffusion process we obtain is more general.

\section{SDE for the first $n$ coordinates and independence}

Now let us focus on determining the SDE for process $\left(X_t\right)_{t \ge 0}$. From \eqref{SDE} we get that $\left(X_t\right)_{t \ge 0}$ satisfies the following
\begin{equation*}
	dX_t= \left(I-\frac{1-\frac{|Y_t|}{R_t}}{|X_t|^2} X_t X_t^\top\right)A'_t dB_t + b(X_t) dt,
	\hspace{0.2cm}\text{where}\hspace{0.2cm}
	A'_t:= \begin{pmatrix}
		I-\frac{1-\frac{|Y_t|}{R_t}}{|X_t|^2} X_t X_t^\top, & -\frac{c^2}{R_t|Y_t|} X_t Y_t \\
	\end{pmatrix}.
\end{equation*}

As before let $(\Tilde{B}'_t)_{t \ge 0}$ be  defined by $\Tilde{B}'_t=\int_0^t A'_s dB_s$. Since $A'_t A'_t{}^\top =I$ implies that $\langle \Tilde{B}'_i,\Tilde{B}'_j\rangle _t=\int_0^t (A'_sA'_s{}^\top)_{ij} ds  = \delta_{ij}t$ for all $i,j \in \{1,2,\ldots, n \}$ by Levy’s characterization we conclude that $(\Tilde{B}'_t)_{t \ge 0}$ is a standard BM on $\mathbb{R}^n$. From this, using $R_t=\sqrt{|X_t|^2c^4+Y_t^2}= \sqrt{(c^4-c^2)|X_t|^2+c^2}$, we get
\begin{equation*}
	dX_t= \left(I-\frac{1-\sqrt{\frac{1-|X_t|^2}{(c^2-1)|X_t|^2+1}}}{|X_t|^2} X_t X_t^\top \right)d\Tilde{B}'_t -\frac{X_t}{2}\frac{(n-1)(c^4-c^2)|X_t|^2+nc^2}{\left( (c^2-1)|X_t|^2+1\right)^2}dt.
\end{equation*}
\begin{remark}
	Notice that in special case when when the ellipsoid is a sphere we have
	\begin{equation*}
		dX_t= \left(I-\frac{1-\sqrt{1-|X_t|^2}}{|X_t|^2} X_t X_t^\top \right)d\Tilde{B}'_t  -\frac{n}{2}X_tdt
	\end{equation*}
	which is a result obtained by Proposition 1.1 in \cite{Mijatovic-Mramor-Uribe-2018}
\end{remark}

Define the process $\left(V_t\right)_{t \ge 0}$ by $V_t:=(V^1_t, V^2_t, \ldots, V^n_t)^\top$, where $V^i_t=\frac{X^i_t}{|X_t|}$ for $i \in \{1,2,\ldots,n\}$. Using Ito's formula we conclude that 
\begin{equation*}
	dV_t= \frac{1}{|X_t|}\left(I- V_t V_t^\top\right)d\Tilde{B}'_t  -\frac{n-1}{2}\frac{1}{|X_t|^2} V_tdt.
\end{equation*}
Next, we introduce the process $\left(\hat{V}_t\right)_{t \ge 0}$ by $\hat{V}_t=V_{T_t}$. Since $|X_t|^2= \frac{c^2-Y_t^2}{c^2}$ the change-of-time formulae for the Ito and Lebesgue-Stieltjes integrals from Chapter V, §1 in \cite{Revuz-Yor-1999} imply that $\hat{V}_t=(I-\hat{V}_t\hat{V}_t^\top)dW_t-\frac{n-1}{2}\hat{V}_tdt$, where $W_t=\int_0^{T_t}\frac{c}{\sqrt{c^2-Y_s^2}}d\Tilde{B}'_s$ is a standard BM on $\mathbb{R}^n$ by Levy's characterization. This implies that the process $\left(\hat{V}_t\right)_{t \ge 0}$ is a BM on $\mathbb{S}^{n-1}$. 

Let us prove that process $(\hat{V}_t)_{t \ge 0}$ is independent of process $(Y_t)_{t \ge 0}$. To do this we first need to modify the BM driving the SDE for $(\hat{V}_t)_{t \ge 0}$. Let us enlarge the probability space to accommodate another scalar BM $(\eta_t)_{t \ge 0}$ which is independent of $(B_t)_{t \ge 0}$ and define a continuous local martingale 
\begin{equation*}
	\hat{W}_t= \int_0^{T_t}\frac{c}{\sqrt{c^2-Y_s^2}}(I-V_sV_s^\top )d\Tilde{B}'_s+ \int_0^{T_t}\frac{c}{\sqrt{c^2-Y_s^2}}V_s d\eta_s.
\end{equation*}
Since $\langle \hat{W}^i, \hat{W}^j\rangle_t=\delta_{ij}t$, we conclude that $(\hat{W}_t)_{t \ge 0}$ is $\mathcal{G}_t$-BM, where $\mathcal{G}_t:= \mathcal{F}_{T_t}$ ($(B_t)_{t \ge 0}$ is $\mathcal{F}_t$-BM). Since
\begin{equation*}
	dV_t=\left(I- V_t V_t^\top\right)\frac{c}{\sqrt{c^2-Y_t^2}}\left[\left(I- V_t V_t^\top\right)d\Tilde{B}'_t + V_t d \eta_t\right]-\frac{n-1}{2}\frac{c^2}{c^2-Y_t^2} V_tdt
\end{equation*}
we can use change of time formula for stochastic and Lebesgue-Stieltjes integral from Chapter V, §1 in \cite{Revuz-Yor-1999} and we get $\hat{V}_t= (I-\hat{V}_t\hat{V}_t^\top)d\hat{W}_t-\frac{n-1}{2}\hat{V}_tdt$ so that $\left(\hat{V}_t\right)_{t \ge 0}$ is a BM on $\mathbb{S}^{n-1}$.

To prove independence of $\left(\hat{V}_t\right)_{t \ge 0}$ and $\left(Y_t\right)_{t \ge 0}$ it is enough to prove independence of BMs $\left(\hat{W}_t\right)_{t \ge 0}$ and $\left(\Tilde{B}_t\right)_{t \ge 0}$ which are driving the respective SDEs. We define $\mu_t := \Tilde{B}_{T_t}= \int_0^{T_t}A_s dB_s$ so that $\mu$ is a $\mathcal{G}_t$-local martingale. Since $\langle \mu\rangle _t= T_t$, the inverse of $T$ is $S$, $\langle \hat{W}^i, \hat{W}^j\rangle _t=\delta_{ij}t$ for each $i,j\in \{1,2,\ldots, n\}$ and for each $i\in \{1,2,\ldots, n\}$
\begin{equation*}
	\langle \hat{W}^i, \eta\rangle _t= \int_0^{T_t}\frac{c^2}{c^2-Y_s^2}\sum_{k=1}^n\sum_{j=1}^{n+1}(\delta_{ik}-V^i_sV^k_s )A'^{kj}_sA^j_sds=0,
\end{equation*}
we can use Knight’s Theorem (also known as the multidimensional Dambis-Dubins-Schwarz Theorem) Chapter V, Theorem 1.9 in \cite{Revuz-Yor-1999} to show that $(\hat{W}_t)_{t \ge 0}$ and $(\mu_{S_t})_{t \ge 0}=(\Tilde{B}_t)_{t \ge 0}$ are independent BMs. This completes the proof of Proposition \ref{SPD}.

\section*{Acknowledgements}
We thank Aleksandar Mijatović, Dario Spanò and Jaromir Sant for useful discussions. Financial support through the \textit{Croatian Science Foundation} under project 8958 and \textit{London Mathematical Society} under Grace Chisholm Young Fellowship are gratefully acknowledged. We also thank the anonymous referee for the helpful comments.
	
\bibliographystyle{alpha}
\bibliography{references.bib}
	
\end{document}